\documentclass{amsart}

\usepackage{amssymb}
\usepackage{amsmath,amsfonts,amsthm,amstext}
\usepackage{amsmath}
\usepackage{mathrsfs}  
\usepackage[all]{xy}
\usepackage{xcolor}
\usepackage{tikz-cd}

\newtheorem{theorem}{Theorem}[section]
\newtheorem{lemma}[theorem]{Lemma}

\newtheorem{prop}[theorem]{Proposition}
\newtheorem{cor}[theorem]{Corollary}

\newcommand{\Z}{\mathbb{Z}}
\newcommand{\z}{\mathbb{Z}}

\theoremstyle{definition}

\newcommand{\verticalinfty}{\rotatebox{135}{$\infty$}}

\begin{document}

		\title[]{
		Hopfian combinatorial wreath products} \author{ Dessislava H. Kochloukova}
	\address{State University of Campinas (UNICAMP), SP, Brazil \\
		\newline
		email :  desi@unicamp.br
		} 
	\email{}
	\date{}

\begin{abstract} Let $A$ be an abelian group. We consider sufficient conditions for the combinatorial wreath product $A \wr_X B$ to be Hopfian generalising results of Bradford and Fournier-Facio. For an integer $m \geq 2$ we show an example where $\mathbb{Z}/ \mathbb{Z}_m \wr_X B$ is not Hopfian but $B$ is  Hopfian. We describe  $Aut(A \wr_X B)$ under some restrictions on $A$, $B$ and $X$.

\end{abstract}

\maketitle

\section{Introduction}

Let $G =  A \wr_X B$ be a combinatorial wreath product. Here $X$ is a set acted by $B$ on the left. All actions and modules in this paper are left ones. If $X = B$ is the $B$-action given by multiplication on the left then $A \wr_X B $ is the restricted wreath product $A \wr B$. The group $G =  A \wr_X B$ was considered in \cite{Cor2} where Cornulier classified when $G$ is finitely presented ( in terms of generators and relations). This was extended to type $FP_n$ in the work of Bartholdi, Cornulier and Kochloukova \cite{B-C-K} and Kropholler and Martino \cite{Kr-Ma}.
In \cite{Cor2} Cornulier classified when  $G$  is residually finite  extending a result of Gruenberg on restricted wreath product \cite{Gru}. 

Recall that a group is Hopfian if it does not have proper quotients isomorphic to it.
Recently Bradford and Fournier-Facio classified in terms of the  Kaplansky Conjecture when $G =A \wr B$ is Hopfian for $A$ a finitely generated abelian group  \cite{Br-FF}. Recall that the Kaplansky direct (resp. stable) finiteness conjecture states that for any group $C$ and  a field $K$ the group algebra $KC$ (resp.  the matrix ring  $M_n(KG)$ for arbitrary $n \geq 1$) is directly finite, i.e. every right inverse is a left inverse. It is known that the stable  finiteness of $KG$ is equivalent to the direct finiteness of $K[H \times G]$ for all finite groups $H$  \cite{D-J} and that the Kaplansky direct finiteness conjecture holds for any field $K$ of characteristic zero \cite{Ka}. The case when $K$ is a field of prime characteristics is still open though the result is known for sofic groups \cite{E-S} and more generally surjunctive groups \cite{Ph}.

Our first result is in contrast to the main result of \cite{Br-FF} where  it was shown that for $B$  a Hopfian  group the group $G = \mathbb{F}_p \wr B$ is Hopfian if and only if the Kaplansky conjecture holds for the ring $\mathbb{F}_p B$ and by now there is no counter example of the Kaplansky conjecture in positive characteristic. 

\medskip
{\bf Proposition A} [Prop. \ref{example}, Cor. \ref{cor1}]
{\it a) 
 Let  $m \not= -1$ be an integer, $K$ a commutative ring with $m+1 \in K^*$. Consider the  HNN extension  $$B = \langle h_0, t \ | \ t^{-1} h_0 t = h_0^{m+1} \rangle \hbox{ and }  H = \langle h_0 \rangle$$
Then the ring
  $End_{K B} ( Ind_H^B ( K))$ is not  directly finite.
  
b) Let $G= \frac{\mathbb{Z}}{m \mathbb{Z}} \wr_X B$ and $X = B/H = \{ b H \ | \ b \in B \}$ be one $B$-orbit. Then $B$ is Hopfian but  $G$ is not Hopfian.
}

\medskip
Note that for the group $G$ from  Proposition  A by the Cornulier criterion of residual finiteness  \cite{Cor} $G$ is not residually finite since $H$ is not closed in the profinite topology of $B$.    This is consistent with the fact that  finitely  generated residually finite groups are Hopfian  \cite{Mal} and our group $G$ is not Hopfian. Note that the group $B$ from Proposition A  is a  metabelian Baumslag-Solitar group, by a result of Hall  finitely generated metabelian groups are residually finite \cite{Hall2}, hence Hopfian.

As in the case of restricted wreath product considered in  \cite{Br-FF}, we show in Theorem B that  whether $G$  is Hopfian  for $G = A \wr_X B$, $A$  a finitely generated, abelian group and under some natural restrictions on $B$ and $X$, depends on some matrix rings over endomorphism rings being directly finite.

We state some definitions needed in the statement of Theorem B.
For a conjugacy class $^B b = \{ g b g^{-1} \ | \ g \in B \}$ we say it is non-abelian if it contains at least two  elements that do not commute. 
For a finitely generated abelian group $A$ using additive notation for $A$ the subgroup $A_p$ denotes the $p$-primary part of $A$ and $A_0$ denotes  $A / tor(A)$.

\medskip
{\bf Theorem B}{\it  
Let $G = A \wr_X B$ be a group,  $D$ be the kernel of the action of $B$ on $X$. Assume that

1) $B$ is a  Hopfian group, 

2) $A$ is a finitely generated abelian group  and if $A$ has exponent 2 then $D$ does not have elements of order 2, 

3)   $Aut(B)$ permutes $\{ stab_B(x) \}_{x \in X}$,

4) every non-trivial conjugacy class in $B$ of an element of $D$ is non-abelian.

Then the following are equivalent:

a) $G$ is a Hopfian group,

b) the ring $End_{\Z B} (A X)$ is directly finite,

c) for every $p$ such that $A_p \not= 0$ we have

c1) if $p$ is prime then $M_{n_p}(End_{\mathbb{F}_p B} (\mathbb{F}_p X))$ is directly finite where $n_p = dim_{\mathbb{F}_p} A_p/ p A_p$,

c2) if $p = 0$ then $M_{n_0}(End_{\Z B} (\Z X))$ is directly finite, where $n_0 = rk A_0$.}

\medskip

Corollary C  is an aplication of the main result of \cite{Lu}
where Lundström gives some sufficient conditions for  $End_{{K} B}( K X)$ to be directly  finite, where $K$ is a field of characteristic zero.  Lundström's sufficiency conditions are $B \backslash X$ is finite and condition 4) from Corollary C.
 The idea behind  the proof  in \cite{Lu} is to use the generalization  of Passman's algebraic proof of the Kaplansky conjecture in characteristic zero from  \cite{Pass}  to a more general setting  \cite{Al}.

\medskip
{\bf Corollary  C} {\it 
Let $G = \mathbb{Z}^n \wr_X B$ be a group, where $B$ acts on $X$ with finitely many orbits. Let $D$ be the kernel of the action of $B$ on $X$. Assume that

1) every non-trivial conjugacy class in $B$ of an element of $D$ is non-abelian,

2) $B$ is a  Hopfian group,  

3)    $Aut(B)$ permutes $\{ stab_B(x) \}_{x \in X}$,

4) $[stab_B(x) : stab_B(x) \cap stab_B(y)] =  [stab_B(y) : stab_B(x) \cap stab_B(y)]$ for all $x,y \in X$.

Then  $G$ is Hopfian.
}

\medskip

 Obviously  if $stab_B(x)$ and $stab_B(y)$ are finite groups  condition 4) holds if and only if $|stab_B(x)| = | stab_B(y)|$.   Still we cannot apply Lundström  result if $A$ contains torsion or the stabilizers are all finite but not of the same order. Using a slightly different approach we deal in the following corollary  with the case when $A$ has torsion and all stabilisers are finite.

\medskip
{\bf   Corollary D} {\it Let  $G = A \wr_X B$ be a group,  where $B$ acts on $X$ with finitely many orbits. Let $D$ be the kernel of the action of $B$ on $X$.
 Assume that

1)  every non-trivial conjugacy class in $B$ of an element of $D$ is non-abelain, 

2) $B$ is a  Hopfian group, A is a finitely generated abelian group and if $A$ has exponent 2 then $D$ does not have elements of order 2,

3)    $Aut(B)$ permutes $\{ stab_B(x) \}_{x \in X}$,

4)  one of the following holds :

4.1) for every $x \in X$ the stabiliser  $stab_B(x)$ is closed in the profinite topology of $G$, 

4.2)   for every $x \in X$   the stabiliser $stab_B(x)$ is finite and for every prime $p$  such that $A_p \not= 0$ the group $stab_B(x)$ is finite    of order coprime to $p$  and $\mathbb{F}_p B$ is  stably finite.

Then  $G$ is Hopfian.
}

\medskip
The group  $G = \frac{\Z}{m \Z} \wr_X B$ from  Proposition A   satisfies 1) (actually the action is faithful), 2), 3) but not 4) as the stabiliser $H \simeq \Z$ is neither finite nor closed in the profinite topology of $B$.  We think of this group as a minimal example where Corollary D is not aplicable for $\mathbb{F}_p \wr_X B$ and $\Z \wr_X B$. Still we do not know whether the group $\Z \wr_X B $ with $X$ and $B$ defined as in  Proposition A is Hopfian as we do not know whether $End_{\Z B} (\Z X)$ is directly finite in this case.

We state some applications  of Corollary D in section \ref{applications}.

The methods used in the proofs of the previous results naturally lead to a description of  $Aut(A \wr_X B)$ we state in Proposition \ref{decompose}. We note that $ Aut(A \wr B)$ was studied by   Houghton in \cite{Ho} and by Genevois and Tessera in \cite{G-T}. Some important subgroups of $Aut(A \wr_X B)$ were studied by Mohammadi Hassanabadi  in \cite{Mo} and
a  description of $Aut(A \wr_X B)$ was obtained by Bodnarchuk  in \cite{Bod} under the assumption that $B$ acts transitively and faithfully on $X$. We do not require that the action of $B$ on $X$ is faithful or transitive but still we impose conditions on the action.

\section{Hopfian and non-Hopfian wreath products}

Recall that a ring $R$ with a unit $1$ is directly finite if for some $a,b \in R$ we have $ab = 1$  then $ba = 1$. By a famous result of Kaplansky \cite{Ka}, see \cite{Mo} for $C^*$-algebra proof, if $K$ is a field of characteristic 0 and $G$ is an arbitrary group then the group algebra $KG$ is directly finite.  Furthermore $M_n(KG)$ is directly finite, which is equivalent to every epimorphism $(KG)^n \to (KG)^n$ of $KG$-modules should be an isomorphism.

\begin{theorem}\cite{Lu}  \label{beautiful} Let $B$ acts on $X$ with finitely many orbits, $B(x) = stab_B(x)$ and for every two $x,y \in X$ assume that $$[B(x) : B(x) \cap B(y)] = [B(y) : B(y) \cap B(x)]$$ Then for a field $K$ of characteristic zero   the ring $End_{K B}( K X)$ is directly finite.
\end{theorem} 
In the above theorem the equality $[B(x) : B(x) \cap B(y)] = [B(y) : B(y) \cap B(x)]$ does not force the numbers to be finite but if one is finite the other is finite and equal to it.

We will show that Theorem \ref{beautiful} does not hold if the condition   $[B(x) : B(x) \cap B(y)] = [B(y) : B(y) \cap B(x)]$ is removed. In the following proposition $t^{-1} H t \cap H = t^{-1} H t$ has index $m+1$ in $H$ and index 1 in $t^{-1} H t$.

Recall that for a ring $R$ the group of invertible  elements is denoted by $R^*$.

\begin{prop}  \label{example} Let  $m \geq 2$ be an integer, $K$  a commutative ring with $m+1 \in K^*$  and consider the following HNN extension  $$B = \langle h_0, t \ | \ t^{-1} h_0 t = h_0^{m+1} \rangle \hbox{ and }  H = \langle h_0 \rangle$$
Then the ring
  $End_{K B} ( Ind_H^B ( K))$ is not  directly finite.
\end{prop}

\begin{proof} Consider $V = Ind_H^B ( K)  = K[ B/H] = KB v$, where $hv = v$ for $h \in H$. Note that $V \simeq KX$, where $X = B/H$.

Consider $\alpha, \beta \in End_{K B} ( Ind_H^B ( K))$  defined by $$\alpha (v) =  t v \hbox{ 
and }\beta(v) = (\sum_{0 \leq j \leq m} h_0^j)t^{-1} v$$They are well defined since for $h \in H$ we have $$h t v = t t^{-1} h t v  = t h^{m+1} v  =t v $$ and for $h = h_0^{(m+1)s + j_0}$ for some $ j_0 \in \{ 0,1, \ldots , m \}$ we have  $$h (\sum_{0 \leq j \leq m} h_0^j)t^{-1} v  =  (\sum_{0 \leq j \leq m} h_0^j) ht^{-1} v  =   (\sum_{0 \leq j \leq m} h_0^j)  h_0^{j_0}  h_0^{(m+1)s} t^{-1} v  =$$ $$(\sum_{0 \leq j \leq m} h_0^j)  h_0^{j_0} t^{-1}   h_0^{s}  v  =
(\sum_{0 \leq j \leq m} h_0^j)  h_0^{j_0} t^{-1}    v  = (\sum_{0 \leq j \leq m} h_0^j)  t^{-1}    v$$
since $h_0^{m+1+i} t^{-1} v =h_0^i t^{-1} th_0^{m+1} t^{-1}v = h_0^i t^{-1} h_0v =  h_0^i t^{-1} v$.

Then
$$\alpha \beta (v) = \alpha ( (\sum_{0 \leq j \leq m} h_0^j)t^{-1} v)   =(\sum_{0 \leq j \leq m} h_0^j)   t^{-1} \alpha (v) =$$ $$ (\sum_{0 \leq j \leq m} h_0^j)   t^{-1} t v = (\sum_{0 \leq j \leq m} h_0^j) v =
(m+1) v$$ and
$$\beta \alpha (v) = \beta ( t v ) = t \beta(v)  = t(\sum_{0 \leq j \leq m} h_0^j)t^{-1} v =
(\sum_{0 \leq j \leq m} h_0^{j/ (m+1)}) v$$

Since $m+1 \in K^*$ we get $(\frac{1}{m+1} \alpha) \beta = Id_V$ but $\beta ( \frac{1}{m+1} \alpha) \not= Id_V$.  Indeed  $$ ( 1 - h_0^{1/ ( m+1)})\beta \alpha(v)  = ( 1 - h_0^{1/ ( m+1)}) (\sum_{0 \leq j \leq m} h_0^{j/ (m+1)}) v = (1 - h_0) v = 0$$ but  $ (m+1) ( 1 - h_0^{1/ ( m+1)})v \not= 0$, hence $ \beta \alpha (v) \not= (m+1)v$.

\end{proof}

All modules and actions considered are left ones.

Let $B$ be a group and $X$ a set such that $B$  acts on  $X$ on the left i.e. there is a map $B \times X \to X$ that sends $(b,x)$ to $b*x$ such that $b_1* (b_2 * x) = (b_1 b_2) * x$ and $1 * x = x$.   Thus $$X = \cup_{i \in I} B/H_i  = \cup_{i \in I} B * x_i $$ is a disjoint union with $H_i = stab_B(x_i)$ the stabilizer of $x_i$ in $B$.

Let $A$ be  a group,  $$G = A \wr_X B$$  Note that $G = M \rtimes B$ where
 $M$ is the normal closure of $A$ in $G$ . We will consider only the case when $A$ is abelian, then $M = AX = \oplus_{x \in X} Ax$, where $Ax \simeq A$, and for $m= \sum_{x \in X} a_x x \in M$, $b \in B$ we have $bmb^{-1} = \sum_{x \in X} a_x (b*x)$. 
 Define $$\pi : G \to B$$  the canonical projection with kernel $M = AX $.
 
  We use $\circ$ for the  conjugation action   (on the left) of $B$ on $M$ i.e. $$b \circ m = bm b^{-1}$$ thus $b \circ x = bxb^{-1} = b * x$ for $x \in X, b \in B$.
  The case when $X = B$ with $b_1 * b_2 = b_1b_2$ leads to the definition of the restricted wreath product $G = A \wr B$.
For $X =   B/ H = \{ b H \ | \ b \in B \}$ we have the regular left action of $B$ on $X$ i.e. $b_0 * bH = b_0 b H$.

\begin{cor} \label{cor1} Let $m \geq 2$ be an integer,   $B$ the group defined in Proposition \ref{example},  $G= \mathbb{Z} / m \mathbb{Z} \wr_X B$ and $X = B/ H = \{ b H \ | \ b \in B \}$ is one $B$-orbit. Then $B$ is a Hopfian group but  $G$ is not Hopfian.
\end{cor} 

\begin{proof} Note that $G = M \rtimes B$ for $M = Ind_H^B( \mathbb{Z}/ m \mathbb{Z})$.
Consider $ \theta : G \to G$ the homomorphism that is the identity on $B$ and $\theta |_M = \frac{1}{m+1}\alpha$ where $\alpha$ was defined in the proof of Proposition \ref{example}. Then $ \frac{1}{m+1} \alpha \beta  = Id_M$ implies $M \leq Im (\theta)$, hence $\theta$ is surjective. But $0 \not=  \frac{1}{m+1}\beta \alpha (v) - v \in Ker (\alpha)  = Ker (\theta)$, so $\theta$ is not injective, where $v$ was defined in the proof of Proposition \ref{example}, actually $X = B * v$ with $stab_B(v) = H$.

To prove that $B$ is Hopfian observe that $B$ is  metabelian and finitely generated hence by a result of Philip Hall is residually finite \cite{Hall2}. But finitely generated residually finite groups are Hopfian \cite{Mal}.
\end{proof}

A non-trivial conjugacy class $^B b$ ( i.e. $b \not= 1$) is called non-abelian if it contains two elements that do not commute.

\begin{lemma} \label{tech} 
Let $G = A \wr_X B$ be a group where $A$ is abelian  and $\theta: G \to G$ an epimorphism. Let $D$ be the kernel of the action of $B$ on $X$. Assume that  

1) every non-trivial conjugacy class in $B$ of an element of $D$ is non-abelian.

2) if $A$ has exponent 2 then $D$ does not have elements of order 2.

Then

a) $\theta(M) \subseteq M$,

b) If furthermore $B$ is Hopfian, then $\theta(M) = M$ and $\sigma = \pi (\theta |_B) : B \to B$ is an isomorphism,  where $\pi : G \to B$ is the canonical projection with kernel $M$.

\end{lemma}

\begin{proof}
a)
Let $m_0 \in M = AX, b_0 \in B$ be such that $m_0 b_0 \in \theta(M)$. Since $\theta$ is surjective and $M$ is an abelian normal subgroup of $G$ we conclude that $\theta(M)$ is an abelian normal subgroup of $G$,  hence the normal subgroup of $G$ generated by $m_0 b_0$ is abelian. Let $m_1\in M$ then the elements
$ ^{m_1} (m_0 b_0)$ and $ m_0 b_0 \hbox{ commute}$.
Note that
$$
(^{m_1} (m_0 b_0)) m_0 b_0 = m_1  m_0 (^{b_0 } m_1^{-1}) b_0    m_0   b_0 = 
m_1  m_0 (^{b_0 } m_1^{-1})   (^{b_0} m_0)   b_0^2$$
and  
$$
m_0 b_0(^{m_1} (m_0 b_0)) =   m_0 b_0  m_1  m_0 (^{b_0} m_1^{-1})  b_0 = 
 m_0 (^{b_0}  m_1)  (^{b_0} m_0) (^{b_0^2} m_1^{-1})  b_0^2$$
Hence the right hand sides of the above two equalities are equal, that together with the fact that conjugates of $m_0, m_1$ commute implies after some cancellation  
$$m_1 (^{b_0 } m_1^{-1})   =  (^{b_0}  m_1)  (^{b_0^2} m_1^{-1})
$$
Using again that conjugates of $m_0$ and $m_1$ commute and 
moving to additive notation in $M$, using $\circ$ for the conjugation action (on the left) of $B$ on $M$ we get
$$  ( b_0 -1)^2 \circ m_1= 0, \hbox{ hence } (b_0 -1)^2 \in Ann_{\Z B} (M)$$

If $A$ is of infinite exponent then $ Ann_{\Z B} (M) = \cap_{i \in I, g \in B} \z B(g H_i g^{-1} - 1)$
where $X = \cup_{i \in I} B/ H_i$ is the decomposition of $X$ as a disjoint union of $B$-orbits.
Since $(b_0 -1)^2 \in  \z B (g H_i g^{-1} - 1) \subset \Z B$ is equivalent to $b_0  \in g H_i g^{-1}$, we get
$$b_0 \in \cap_{i, g}  g H_i g^{-1} = D$$
where $D$ is the kernel of the action of $B$ on $X$.

If $b_0  \not= 1$, note that  since $m_0 b_0 \in \theta(M)$ and $ \theta(M)$  is normal in  $G$  then for any 
$b_1 \in B$ the elements $^{b_1}(m_0 b_0)$ and $m_0 b_0$ belong to the abelian group $\theta(M)$, hence they commute. Working modulo $M$ we get that $^{b_1}b_0$ commutes with $b_0$ a contradiction with $b_0 \not= 1$.

If $A$ is an abelian group of exponent $n> 1$ then $ Ann_{\Z B} (M) = \cap_{i, g \in B} (\z B(g H_i g^{-1} - 1) + n \Z B)$. Set $\mathbb{F}_n = \Z / \Z_n$. If $n > 2$ since  $(b_0 -1)^2 \in  \z B (g H_i g^{-1} - 1) + n \Z$ is equivalent to $(b_0 -1)^2 \in  \mathbb{F}_n B (g H_i g^{-1} - 1)$, hence
$b_0  \in g H_i g^{-1}$, we can proceed as before to get $b_0 \in \cap_{i, g}  g H_i g^{-1} = D$ and $b_0 = 1$. 

If $n = 2$ we have  $(b_0 -1)^2 \in  \z B (g H_i g^{-1} - 1) + 2 \Z B $ is equivalent to $b_0^2 -1 \in \mathbb{F}_2 B (g H_i g^{-1} - 1) \subset \mathbb{F}_2 B$. Hence $b_0^2 \in D$ and 
 $^{b_1}((m_0 b_0)^2)$ commutes with $(m_0 b_0)^2$ in $G$, so $^{b_1} ( b_0^2)$ commutes with $b_0^2$ in $B$, as before this implies  $b_0^2 = 1$. Since in this case $D$ does not have elements of order 2 we have $b_0 = 1$.

\medskip
b) Let $\sigma = \pi (\theta |_B) : B \to B$ where $\pi : G \to B$ is the canonical projection with kernel $M$. Then since $\theta$ is surjective and $\theta(M) \subseteq M$ we have that $\sigma$ is surjective. Since $B$ is Hopfian, $\sigma$ is an isomorphism. Since $\theta$ is surjective for $m \in M$ there are $\widetilde{m} \in M, \widetilde{b} \in B$ 
such that $m = \theta( \widetilde{m}) \theta(\widetilde{b}) \in M \theta(\widetilde{b}) = M \sigma(\widetilde{b})$, so $m \in M \cap M \sigma(\widetilde{b})$, hence $\sigma(\widetilde{b}) = 1$ and $\widetilde{b} = 1$. Thus $M = \theta(M)$ as claimed.
\end{proof}

\begin{lemma} \label{decompose0} 
Let $G = A \wr_X B$ be a group. Let $D$ be the kernel of the action of $B$ on $X$. Assume that

1) every non-trivial conjugacy class in $B$ of an element of $D$ is non-abelian,

2) $B$ is Hopfian,

3)   $Aut(B)$ permutes $\{ stab_B(x) \}_{x \in X}$,

4) $A$ is  abelian and if $A$ has exponent 2 then $D$ does not have elements of order 2.

Then  every epimorphism $\theta : G \to G$ could be decomposed as $$\theta = \theta_2 \theta_1$$ where

i)  $\theta_2 |_M \in End_{\Z B} ( A X)$, $\pi \theta_2 |_B = Id_B$, where $\pi : G \to B$ is the canonical projection with kernel $M$.

ii)  $\theta_1 |_X$  is a permutation of $X$, $\theta_1 |_B = \sigma \in Aut(B)$,  thus $\theta_1 \in Aut(G)$.
\end{lemma}

\begin{proof} 
By Lemma \ref{tech} $\theta(M) = M$ and $ \sigma = \pi \theta|_B$ is an isomorphism.  Recall that $M = A X$ and the action of $B$ on $X$ via conjugation in $G$ is the original action of $B$ on $X$. 
  Let $X = \cup_{i \in I} B/H_i  = \cup_i B * x_i $ be the decomposition as a disjoint union of $B$-orbits.
By condition 3)  $$\sigma(H_i) =  g_i H_{\pi(i)} g_i^{-1}$$ for some $g_i \in B$ and a permutation $\pi$ of $I$. Define a homomorphism $\theta_1 : G \to G$ such that $$\theta_1(x_i) =g_i * x_{\pi(i)}  \in X \subset M\hbox{ and }\theta_1 |_B = \sigma$$ i.e. for $g \in B$ we have  $\theta_1(g * x_i) = \theta_1(^g x_i) = \theta_1(g) \theta_1(x_i) \theta_1(g)^{-1}  = \sigma(g) *(  g_i * x_{\pi(i)})  = ( \sigma(g) g_i)  * x_{\pi(i)}$. In particular   we get that $\theta_1$ is well defined since for $g \in H_i$ we have $\theta_1(x_i) = \theta_1(g * x_i)  = ( \sigma(g) g_i)  * x_{\pi(i)} \in  g_i H_{\pi(i)} g_i^{-1}  g_i *  x_{\pi(i)} =  g_i H_{\pi(i)} *  x_{\pi(i)} = g_i *  x_{\pi(i)} $.
It is easy to see that $\theta_1$ is an automorphism of $G$. 

Then $\theta_2 = \theta (\theta_1)^{-1} : G \to G$ is an epimorphism such that for every $b \in B$ we have $\theta_2 (b)  \in  M b$. By Lemma \ref{tech} $\theta_2(M) = M$ and $\theta_2 |_M$ is a homomorphism of  $\Z B$-modules since $\theta_2 (b \circ m ) = \theta_2(^b m) = ^{\theta_2(b)} \theta_2(m) = ^b \theta_2(m) = b \circ \theta_2(m) $.
 \end{proof}

Let $A$ be a finitely generated abelian group.  We denote by $A_p$  the $p$-primary part of $A$ and by $A_0$ we denote  $A / tor(A)$, where $tor(A)$ is the torsion-free part of $A$.

\medskip
{\bf Proof of Theorem B}
b) implies a) If $\theta : G \to G$ is an epimorphism then by Lemma \ref{tech} $\theta(M) = M$ and $ \sigma = \pi \theta|_B$ is an isomorphism.  

By Lemma \ref{decompose0}  $\theta = \theta_2 \theta_1$ where $\theta_2 |_M \in End_{\Z B} ( A X)$, $\theta_2 |_B = Id_B$ and  $ \pi \theta_1 |_B = \sigma \in Aut(B)$, $\theta_1 \in Aut(G)$.
 Since $\theta$ is an epimorphism then $\theta_2 |_M : M \to M$ is an epimorphism of $\mathbb{Z} B$-modules, hence $\theta_2 |_M$ is an isomorphism of $\mathbb{Z} B$-modules since  $End_{\Z B}(A X)$ is directly finite. Then $\theta_2$  and $\theta$ are isomorphisms.

a) implies b) Read backwords the proof of b) implies a) .

b) implies c) Note that  $A = \oplus_{p \geq 0} A_p$, where we have fixed a subgroup of $A$ isomorphic to $A_0$ and denoted it by the same name. Then
$A X = \oplus_{p \geq 0} A_p X$ and  for $p > 0$ the subgroup $A_p X$ is the $p$-primary part of $AX$. 

If $\alpha_p \in  End_{\Z B} (A_p X)$ are surjective maps, then $\alpha = \oplus_{p \geq 0} \alpha_p \in  End_{\Z B} (A X)$ is surjective, hence an isomorphism. In particular each $\alpha_p$ is an isomorphism.

c) implies b) 
Let $\alpha \in End_{\Z B} (A X)$ be surjective i.e. an epimorphism of $\Z B$-modules.
Then $\alpha (A_p X) \subseteq  A_p X$ for all $p >0$. 

Let $\pi_0: AX \to A_0 X$ be the map that is identity on $A_0$ and sends each $A_p X$ to $0$ for $p >0$. 
Then the map $\alpha_0 = \pi_0 \alpha |_{A_0 X} \in  End_{\Z B} (A_0 X)$ is surjective and since $End_{\Z B} (A_0 X)$ is directly finite $\alpha_0$ is an isomorphism, hence injective. Then $Im ( \alpha |_{A_0 X}) \cap  \oplus_{p >0} A_p X = 0$. Then since $\alpha$ is surjective, $\alpha_p = \alpha |_{A_p X}$ is surjective for $p > 0$. Since $End_{\Z B} (A_p X)$ is directly finite $\alpha_p$ is an isomorphism for all $p >0$. Hence $\alpha$ is an isomorphism.

c) implies d) First for $p > 0$ note that $A_p$  has a finite filtration $\{ p^j A_p \}_{j \geq 0}$ where each quotient is isomorphic to $V_{p,s} = (\mathbb{F}_p)^s$ for some $s \leq n_p$ with $n_p$ attained,  we conclude that  $End_{\Z B} (A_p X)$ is directly finite if and only if  $End_{\Z B} (V_{p,s} X)$ is directly finite for all $s$ that appear as quotients in the above filtration. But $V_{p,s} X \simeq (\mathbb{F}_p X)^s$ implies $End_{\Z B} (V_{p,s} X) \simeq M_s( End_{\Z B} ({\mathbb{F}_p} X))$. Since for $s_1 < s_2$ we have that if $M_{s_2}( End_{\Z B} ({\mathbb{F}_p} X))$ is directly finite then $M_{s_1}( End_{\Z B} ({\mathbb{F}_p} X))$ is directly finite we need only to include in our conclusions that $M_{n_p}( End_{\Z B} ({\mathbb{F}_p} X)) = M_{n_p}( End_{\mathbb{F}_p B} ({\mathbb{F}_p} X))$ is directly finite.

If $A_0 \not= 0$ then $A_0 \simeq \Z^{n_0}$, $A_0 X \simeq (\Z X)^{n_0}$ and $End_{\Z B} (A_0 X) \simeq M_{n_0} (End_{\Z B} (\Z X))$.

d) implies c) Read backwords  the proof of c) implies d).

\medskip
{\bf Proof of Corollary C}
By Theorem B we have to show that $End_{\Z B} (A X)$ is directly finite, where $A = \Z^n$. Note that $AX = ( \Z X)^n = \Z Y$, where $Y$ is the disjoint union of $n$ copies of $X$. Then we can apply Theorem \ref{beautiful} for $\mathbb{Q} Y$ and conclude that  $End_{\mathbb{Q} B} (\mathbb{Q} Y)$ is directly finite, hence  $End_{\Z B} (\Z Y)$ is directly finite.  This completes the proof.

We give special attention to the case when all  stabilisers are finite groups.

\begin{lemma}  \label{finite02}  Suppose $K$ is a field, $I$ is a finite set,  $V = \oplus_{i \in I} Ind_{H_i}^B ( K)$  where $H_i$ are finite groups such that if the characteristic of $K$ is $p > 0$ then for every $i \in I$, $p$ does not divide the order of $H_i$. Suppose that the ring $KB$ is  stably finite. Then
$End_{KB} (V)$ is stably finite.
\end{lemma}

\begin{proof} In the case when $K$ has characteristic 0 we know that $KB$ is stably finite. 
Note that
 $Ind_{H_i}^B ( K) \simeq K B \otimes_{K H_i} K$ is a direct summand of $K B$, since
$$K B \simeq  K B e_{H_i} \oplus K B (1 - e_{H_i})$$
where $e_{H_i} = \frac{1}{| H_i |} \sum_{h \in H_i} h$ is an idempotent and  $Ind_{H_i}^B ( K)  \simeq  K B e_{H_i}$.
Then for $|I| = s$,
$V = \oplus_{i \in I}  Ind_{H_i}^B ( K)$ is a direct summand of $(K B)^s$, hence
$$M_n(End_{K B} (V)) \simeq End_{K B} (V^n) \leq End_{K B} ( (KB)^{sn}) = M_{sn} (KB)$$
i.e. the left hand side is a subring of the right hand side. But $M_{sn} (KB)$ is directly finite when  $KB$ is  stably finite.
\end{proof}

\begin{cor} \label{limit} 
Let  $G = A \wr_X B$ be a group, where $B$ acts with finitely many orbits on $X$.
Let $D$ be the kernel of the action of $B$ on $X$.
Assume that

1)  every non-trivial conjugacy class in $B$ of an element of $D$ is non-abelian,

2) $B$ is  a Hopfian group, $A$ is a finitely generated  abelian group and if $A$ has exponent 2 then $D$ does not have elements of order 2,

3)    $Aut(B)$ permutes $\{ stab_B(x) \}_{x \in X}$,

4)  there exists a nested set $\mathcal{U} = \{ U \}$ of normal subgroups of $B$ such that   $$\cap_{U \in \mathcal{U}}  stab_B(x) U = stab_B(x) \hbox{ for }x \in X$$ and one of the following holds

4.1) for all $U \in \mathcal{U}$ the index $[B : U] <  \infty$,

4.2) for all $U \in \mathcal{U}$,  $x \in X$   the groups $stab_B(x) U/ U$ are finite and for every prime $p$  such that $A_p \not= 0$ the groups $stab_B(x) U/ U$ are finite    of order comprime to $p$  and $\mathbb{F}_p [B/U]$ is  stably finite.

Then  $G$ is Hopfian.

\end{cor}

\begin{proof}
Let  $X = \cup_{i \in I}  B /H_i =  \cup_{i \in I} \{ bH_i \ | \ b \in B  \}$. 

Suppose $\alpha, \beta \in End_{\mathbb{Z}B} (A X)$ be such that $\alpha \beta = Id_{AX}$. Let $\alpha_U, \beta_U$ be the images of $\alpha, \beta$ in $End_{\mathbb{Z}[B/U]} (A X_U)$ where $X_U = U \backslash X = \cup_{i \in U} B/ H_i U$. Then   $\alpha_U \beta_U = Id_{AX_U}$.
We claim that $End_{\mathbb{Z}[B/U]} (A X_U)$ is directly finite, hence 
  $\beta_U \alpha_U =  Id_{A X_U}$.
 Then note that $X_U$ is a quotient of $X$ and this induces a map  $$X \to \varprojlim_{U \in \mathcal{U}} X_U$$ that is injective since  $\cap_{U \in \mathcal{U}}   H_i U = H_i$ for all $i \in I$. Then the induced map
 $$  End_{\mathbb{Z} B } (AX) \to \varprojlim_U   End_{\mathbb{Z}[B/U]} (A X_U)$$ is injective and
 $\beta \alpha$  and $\alpha \beta$ map to the same element $\prod_U (\beta_U \alpha_U) = \prod_U ( \alpha_U \beta_U)$ in the inverse limit, so $\beta \alpha = \alpha \beta$.

To prove the claim that $End_{\mathbb{Z}[B/U]} (A X_U)$  is directly finite consider two cases:

1) Suppose first that condition 4.1) holds.
Then $A X_U$ is a finitely generated $\Z$-module, so $A X_U = \Z^m \oplus C$, where $C$ is finite.  Hence  any epimorphism of  abelian groups $ \Z^m \oplus C \to  \Z^m \oplus C$ is an isomorphism, in particular $\alpha_U$ is an isomorphism, hence $\beta_U \alpha_U =  Id_{A X_U}$.

2) Suppose condition 4.2) holds. 
By the equivalence of parts b) and c) of Theorem B it suffices to show that $End_{\Z [B/U]} (\Z X_U) $ is directly finite if $A$ is infinite and  $End_{ \Z [B/U]} (\mathbb{F}_p X_U) = End_{\mathbb{F}_p  [B/U]} (\mathbb{F}_p X_U)$ is directly finite if $A_p \not= 0$. Note that  $End_{\Z [B/U]} (\Z X_U) $ is a subring of $ End_{\mathbb{Q} [B/U]} (\mathbb{Q} X_U) $ and by Lemma \ref{finite02}   $End_{\mathbb{F}_p  [B/U]} (\mathbb{F}_p X_U)$ and $ End_{\mathbb{Q} [B/U]} (\mathbb{Q} X_U) $ are stably finite.  Note that to apply Lemma \ref{finite02} we used that $\mathbb{Q} [B/U]$ is stably finite for every group $B/U$. 
\end{proof}

{\bf Proof of Corollary D} Follows from the previous corollary applied for $\mathcal{U}$ the set of all normal subgroups of finite index in $B$ or for $\mathcal{U}$ the set with unique element the trivial group.  
\section{Applications of Corollary D} \label{applications}

 We list some classes of groups  $G = A \wr_X B$, $A$ abelian for which we can use our results and deduce that $G$ is Hopfian. In all cases $X = \cup_{i \in  I} B/ H_i$ with $I$ finite.

\medskip
I.  $B$ is a hyperbolic group without  non-trivial normal finite cyclic subgroups, the finite  set  $\{H_i \}$  contains all (up to conjugation) finite cyclic subgroups of $B$ of order that belongs to a  fixed finite set of integers. 

 It is known that:

 1)  hyperbolic groups are Hopfian, see the introduction of \cite{F-S}. 
 
 2)  there are only finitely many (up to conjugation) elements of a fixed finite order in $B$ \cite{B-G},\cite{Brady}

3) it is an open problem whether all hyperbolic groups are sofic.

4)   the intersection of all conjugates of all $H_i$ is a finite normal subgroup of $B$, hence is trivial i.e. the action of $B$ on $X$ is faithful.
 
 \medskip
 Then by Corollary D  $\Z^n \wr_X B$ is Hopfian. 
 
 The Kaplansky direct finiteness conjecture holds for any field $K$ for sofic groups \cite{E-S}  and if $B$ is sofic then $A \wr_X B$ is Hopfian for all finitely generated groups $A$.

\medskip
II.
$B$ is an arithmetic group  without non-trivial finite normal subgroups,  finite set $\{ H_i \}$ contains  all ( up to conjugation) finite subgroups of $B$ of some fixed orders. 

It is known that:

1)  $B$ is finitely generated and linear, hence residually finite by Malcev´s result \cite{Mal}, hence Hopfian \cite{Mal}. 

2) there are finitely many conjugacy classes of finite subgroups of $B$ \cite{Borel} . 

3) every residually finite group is sofic, hence $B$ is sofic.

4) since by assumption there is no non-trivial  finite normal subgroup in $B$, the  action of $B$ on $X$ is faithful.

\medskip
Then by Corollary D the group $G = A \wr_X B$ is Hopfian for all finitely generated abelian groups $A$. 

5) If all $H_i$ are closed in the profinite topology of $B$ by \cite{Cor2} $G$ is residually finite.

\section{The automorphism group of $G = A \wr_X B$}

Let $X$ be a set on which $B$ acts (on the left). We say the action is $Aut(B)$-good if there is a homomorphism of groups
$ \psi : Aut(B) \to Perm(X)$
such that 
 \begin{equation} \label{good1}\psi(\sigma) (g * x) = \sigma(g) * \psi(\sigma)(x) \hbox{ for all }\sigma \in Aut(B), g \in B, x \in X \end{equation} 

\medskip
\begin{lemma} \label{Aut-good}   Let $B$ be a group, $X = \cup_{i \in I}B/ H_i = \cup_i B * x_i $ is a disjoint union, $B$ acts on $B/H_i $ via multiplication (on the left). Then  the action of $B$ on $X$ is $Aut(B)$-good if  and only if $ Aut(B)$ permutes the set  $\{ stab_B(x) \ | \ x \in X \}$.

\end{lemma}

\begin{proof} Here $H_i = stab_B(x_i)$. Let $\sigma \in Aut(B)$ and  $\psi : Aut(B) \to Perm(X)$  satisfies (\ref{good1}). Let $\psi( \sigma) (x_i) = g_0 * x_j $ for some $g_0 \in B$. Then   for $h_i \in H_i$ we have
$ g_0 * x_j  = \psi( \sigma) (x_i) = \psi( \sigma) (h_i * x_i) = \sigma(h_i) * \psi(\sigma) ( x_i) =  \sigma(h_i) * ( g_0 * x_j)  =(\sigma(h_i) g_0) * x_j$, hence 
$x_j =  ( g_0^{-1} \sigma(h_i) g_0) * x_j $ and $ g_0^{-1} \sigma(h_i) g_0 \in  H_j$ i.e.  $ g_0^{-1} \sigma(H_i) g_0 \subseteq  H_j$. 

Note that $x_i = \psi(\sigma^{-1})(g_0 * x_j) = \sigma^{-1}(g_0) * \psi(\sigma^{-1}) (x_j)$, so $\sigma^{-1}(g_0^{-1}) * x_i = \psi(\sigma^{-1}) (x_j)$. Applying the same argument as above with $i$ and $j$ swapped, $\sigma$ and $g_0$  changed to $\sigma^{-1}$ and $\sigma^{-1}(g_0^{-1})$ we get $\sigma^{-1}(g_0) \sigma^{-1} (H_j) \sigma^{-1}(g_0^{-1}) \subseteq H_i$ i.e. $g_0 H_j g_0^{-1} \subseteq \sigma(H_i)$. Hence  $g_0 H_j g_0^{-1} = \sigma(H_i)$.

For the converse suppose that  $ Aut(B)$ permutes the set  $\{ stab_B(x) \ | \ x \in X \}$ and $\sigma \in Aut(B)$. If $\sigma(H_i) = stab_B(g_0 * x_j) = g_0 H_j g_0^{-1}$ define $\psi(\sigma) (x_i) = g_0 * x_j$. It is easy to check $\psi$ is well-defined and satisfies (\ref{good1}).  
\end{proof}

 \begin{lemma} \label{easy} Let $B$ be a group acting (on the left) on a set $X$ such that $Aut(B)$ permutes $\{ stab_B(x) \}_{x \in X}$ and $A$ be an abelian group. Then
there is a monomorphism of groups $$\rho : Aut(B) \to Aut(A \wr_X B)$$  such that $\rho(\sigma) |_B = \sigma$, $\rho(\sigma) (ax)  = a  \psi(\sigma)(x) $ for $a \in A, x \in X$ and  $ \psi$ satisfies (\ref{good1}). 
 \end{lemma}

 \begin{proof}
 By Lemma \ref{Aut-good} there is $\psi$ as in (\ref{good1}). It is easy to check that $\rho(\sigma)$ is an automorphism of $A \wr_X B$ and
 $\rho$  is a  homomorphism.
 \end{proof}

 Let $V$ be a left $B$-module.
 A derivation $\gamma : B \to V$ is a linear map such that $$\gamma(b_1 b_2) = \gamma(b_1) + b_1 \circ \gamma(b_2)$$ where the action of $B$ on $V$ is denoted by $\circ$. Denote by $Der(B, V)$ the abelian group of all derivations with group operation traditional sum of functions.
 
 For $m \in AX$ denote by $\gamma_m \in Der(B, AX)$ the principal derivation defined by $\gamma_m(b) = (1 - b) \circ m$.
In a group $H$ denote by $I_{h_0}$ the inner automorphism defined by $I_{h_0}(h) = h_0 h h_0^{-1}$.

Suppose $ \psi : Aut(B) \to Perm(X)$ satisfies (\ref{good1}).
For $b_0 \in B$ define the linear map $$\nu_{b_0} : AX \to AX $$  by $ \nu_{b_0}(ax) = a( b_0 * [\psi (I_{b_0})^{-1}(x)]) \in aX$   for all $x \in X$, $a \in A$.
Let $$\delta : B \to Iso_{\Z B}(A X) $$  be the map that sends $b$ to $\nu_b$. It is easy to check that $\delta$ is an antihomomorphism.
Note that for $b_0 \in B$   \begin{equation} \label{conta2} \nu_{b_0} \in  Iso_{\Z B}(A X). \end{equation} 
Indeed
 $\nu_{b_0} ( b \circ ax) = \nu_{b_0} (  a( b * x)) =a   b_0 * [\psi (I_{b_0})^{-1}( b * x)] = $ $a b_0 * ( b_0^{-1} b b_0 *  [\psi (I_{b_0})^{-1}(  x)])= a(b b_0) *  [\psi (I_{b_0})^{-1}(  x)] =$ $a b * ( b_0 *  [\psi (I_{b_0})^{-1}(  x)]) = b \circ \nu_{b_0} ( ax).
$

 \begin{prop} \label{decompose} 
Let $G = A \wr_X B$ be a group.  Let $D$ be the kernel of the action of $B$ on $X$.
 Assume that

1) A is a finitely generated abelian group and if $A$ has exponent 2 then $D$ does not have elements of order 2,

2)  every non-trivial conjugacy class in $B$ of an element of $D$ is non-abelian, 

3)  $Aut(B)$ permutes $\{ stab_B(x) \}_{x \in X}$.

Then 

a)
\begin{equation} \label{iso-iso} Aut(G) \simeq ( Der(B, AX) \rtimes Iso_{\Z B }(A X)) \rtimes Aut(B) \end{equation}

where $Iso_{\Z B}(AX)$ is the group of  isomorphisms of the (left)  $\Z B $-module $AX$,  in (\ref{iso-iso}) $Aut(B)$ corresponds to $\rho(Aut(B)) \leq Aut(G)$ and  $ Der(B, AX)$ and $ Iso_{\Z B}(A X)$ are invariant under conjugation with $Aut(B)$.

b)   there is  an exact sequence
$$
1 \to ( Iso_{\Z B}(A X))/ \delta(Z(B)) \to Out(G) /  H^1(B, AX) \to Out(B) \to 1$$

\end{prop}

\begin{proof}
a) Let $\theta \in Aut(G)$, by Lemma \ref{tech} $\theta(M) \subseteq M$, similarly $\theta^{-1} (M) \subseteq M$, so $\theta(M) = M$. 
Consider $\sigma = \pi \theta |_B : B \to B$ an epimorphism, where $\pi : G \to B$ is the canonical projection with kernel $M = AX$. We do not assume $B$ is Hopfian but since $\theta \in Aut(G)$ we have  $\sigma \in Aut(B)$. Define  $\theta_1 = \rho(\sigma)$  where $\rho$ is defined in Lemma \ref{easy}. Then
 we can decompose $\theta = \theta_2 \theta_1$ where $\theta_2 |_M \in End_{\Z B} ( A X)$, $\pi \theta_2 |_B = Id_B$. 

For every $b \in B$ we write $\theta_2(b)= \gamma(b)b \in Mb$. Since $\theta_2$ is a homomorphism   
$$
\gamma(b_1 b_2) b_1 b_2 = \theta_2(b_1 b_2) = \theta_2(b_1) \theta_2 (b_2) = \gamma(b_1) b_1 \gamma(b_2) b_2 = \gamma(b_1) ( b_1 \gamma(b_2) b_1^{-1}) b_1 b_2$$
and moving to additive notation in $M$ we have
$\gamma(b_1 b_2)  = \gamma(b_1) + b_1 \circ \gamma(b_2)$
i.e. $\gamma : B \to M$ is a derivation.
Since $\theta_2$ is an isomorphism  $\theta_2 |_M \in Iso_{\Z B} (A X)$.
Thus $\theta_2$ can be described by the pair $(\gamma,\theta_2 |_M)$. 

Consider the group
$$\Delta = \{ \alpha \in Aut(G) \ | \ \pi \alpha |_B = Id_B \}$$
Note that $Aut(G) = \Delta \rho(Aut(B))$ is a semi-direct product with $\rho(Aut(B)) \simeq Aut(B)$, so we have an isomorphism
$Aut(G) \simeq \Delta \rtimes Aut(B)$.

We want to describe the group structure of $\Delta$.
Let $\alpha_1, \alpha_2 \in \Delta$, where  $\alpha_1 = ( \gamma_1, \beta_1),\alpha_2 = ( \gamma_2,\beta_2)$, 
where $\beta_i = \alpha_i |_M$ and $\gamma_i: B \to M$ is a derivation.
Then $\alpha_1 \alpha_2 |_M = \beta_1 \beta_2 = : \beta$ and so $\alpha_1 \alpha_2 = (  \gamma,\beta)$, hence for $b \in B$ we have
$\gamma(b) b = \alpha_1 \alpha_2(b) =\alpha_1( \alpha_2(b)) =\alpha_1( \gamma_2(b) b) = \alpha_1( \gamma_2(b)) \alpha_1(b) =  ( \beta_1 \gamma_2)(b) \gamma_1(b) b
$,
hence moving to additive notation in $M$ we have
$\gamma(b) = ( \beta_1 \gamma_2)(b) + \gamma_1(b)$, 
so the group structure in $\Delta$ is given by
$( \gamma_1,\beta_1) (  \gamma_2,\beta_2) = (\beta_1 \gamma_2  + \gamma_1,\beta_1 \beta_2)$.
Then $\Delta$ can be decomposed as a semi-direct product
$\Delta \simeq  Der(B, AX) \rtimes Iso_{\Z B}(A X)$.

\medskip

b) Consider the inner automorphism $\theta = I_{m_0b_0} \in Inn(G)$, where $m_0 \in M = AX, b_0 \in B$.
  For $m \in M, b \in B$ we have
$ \theta(mb) =  m_0 b_0 m b (m_0 b_0)^{-1} = 
 m_0(^{b_0} m) m_0^{-1} $ $(^{m_0} (^{b_0} b)) =$ $ (^{b_0} m) ( m_0 (^{b_0} b) m_0^{-1})$.
Moving to additive notation in $M$ we get
\begin{equation} \label{conta1}  \theta(mb) =(^{b_0} m) ( m_0 (^{b_0} b) m_0^{-1} (^{b_0} b)^{-1}) (^{b_0} b) = ( b_0 \circ m + m_0 -(^{b_0} b) \circ m_0)   (^{b_0} b) \end{equation}

We decompose $\theta = \theta_2 \theta_1$ where $\theta_1 \in \rho (Aut(B))$,  $\rho$ is defined in Lemma \ref{easy} and $\pi \theta_2 |_B = Id_B$. Note that
$\theta_1 = \rho( \sigma_1)$ where $\sigma_1 = I_{b_0} \in Aut(B)$. Hence for all $b \in B$
$$\theta_1 (b) = ^{b_0} b = b_0 b b_0^{-1} \hbox{
 and }\theta_1( ax) =  a \psi( \sigma_1)(x) \in X$$
Using that $\theta = \theta_2 \theta_1$ and (\ref{conta1})  we have for $b \in B$ 
$$
\theta_2 (b) = \theta \theta_1^{-1} (b) = \theta(b_0^{-1} b b_0) =( (1 - b) \circ m_0) b = \gamma_{m_0} (b) b$$
where $\gamma_{m_0}(b) = ( 1 - b) \circ m_0$.
 For $x_1 \in X = B$ by (\ref{conta1}) we have $ \theta(x_1) = b_0 \circ x_1 = b_0 * x_1$, hence
$\theta_2(ax) = \theta \theta_1^{-1} (ax)  = \theta[ a \psi( \sigma_1^{-1})(x)] =  \theta [a \psi (I_{b_0})^{-1}(x)] =  $
$a[ b_0 \circ \psi (I_{b_0}^{-1})(x)]= a[ b_0 * \psi (I_{b_0}^{-1})(x)]=a \nu_{b_0}(x) = \nu_{b_0}(ax).
$
 Hence $\theta = I_{m_0 b_0}$ corresponds to \begin{equation} \label{feb} (\gamma_{m_0}, \nu_{b_0}, I_{b_0}) \in  (  PDer(B, AX) \rtimes Iso_{\Z B}(A X)  ) \rtimes Inn(B) \end{equation}
where $Inn$ denotes inner automorphisms and  $PDer$ denotes  principal derivations. 

By standard homological algebra  
$Der/ PDer \simeq H^1$ is the first cohomology.
Hence combining (\ref{feb}) with (\ref{iso-iso}) we get the isomorphism 
\begin{equation} \label{important8} Out(G) \simeq  H^1(B, A X) \rtimes ( ( Iso_{\Z B}(A X)  \rtimes Aut(B)) / \Gamma) \end{equation}
where $\Gamma = \{ (\nu_{b_0}, I_{b_0} ) \ | \ b_0 \in B \}$ is   a subgroup of $ Iso_{\Z B}(A X)  \rtimes Aut(B) $.
Note that  $I_{b_0}$ is trivial if and only if  $b_0 \in  Z(B)$, hence   
the  (split) short exact sequence
$1 \to Iso_{\Z B}(A X)  \to  Iso_{\Z B}(A X)   \rtimes Aut(B) \to Aut(B) \to 1$
gives the following (in general non-split) exact sequence
$$
1 \to ( Iso_{\Z B}(A X)  )/ (\Gamma \cap   Iso_{\Z B}(A X)  ) \to (  Iso_{\Z B}(A X) \rtimes Aut(B)) / \Gamma \to Out(B) \to 1$$
Note that $\Gamma \cap   Iso_{\Z B}(A X)  = \{ \nu_{b_0} \ | \ I_{b_0} = 1 \} = \{ \nu_{b_0} \ | \ b_0 \in Z(B) \} $.
Recall that $\delta : B \to Iso_{\Z B }(A X) $  is the antihomomorphism  that sends $b$ to $\nu_b$.
Combining with (\ref{important8}) this yields the (in  general non-split) exact sequence
\begin{equation} \label{important7}
1 \to Iso_{\Z B}(A X)  / \delta(Z(B)) \to Out(G) /  H^1(B, A X) \to Out(B) \to 1 \end{equation}
Finally we note  that it is easy to verify that  $Iso_{\Z B}(AX)$ and  $  Der(B, AX)$  are invariant under conjugation with $Aut(B)$.
\end{proof}

In the next lemma we use the notion of the number $e(B,H)$ of ends of pairs of groups.

\begin{lemma}\label{ends}  Under the hypothesis of Proposition \ref{decompose}   assume further that for $X = \cup_{i \in I} B/H_i$  we have

1) $e(B, H_i) = 1$ and $H_i$ has infinite index in $B$ for each $i \in I$,

2) if $A$ is infinite then for every normal subgroup $N$ of $H_i$ such that $H_i / N \simeq \mathbb{Z}$ we have $e(B, N) = 1$. If the $p$-primary part $A_p$ of $A$ is non-zero  then for every normal subgroup $N$ of $H_i$ of index $p$ we have $e(B, N) = 1$.

Then  $ H^1(B, AX)= 0$ and in particular there is a short exact sequence of groups
\begin{equation} \label{important5}
1 \to ( Iso_{A B}(A X))/ \delta(Z(B)) \to Out(G)  \to Out(B) \to 1\end{equation}
\end{lemma}

\begin{proof}  Note that $AX \simeq \oplus_{i \in I} Ind_{H_i}^B(A)$, where $A$ is the trivial $H_i$-module,
hence
$H^1(B, AX) = \oplus_{i \in I} H^1(B, Ind_{H_i}^B(A))$.
Since $A$ is finitely generated $A \simeq \oplus_{p \geq 0} A_p$ where $A_p$ is the $p$-primary part of $A$ and $A_0 \simeq A  / tor(A)$.
Thus it remains to show $H^1(B, Ind_{H_i}^B(\mathbb{F}_p)) = 0$  if  $A_p \not= 0$  and $H^1(B, Ind_{H_i}^B(\Z)) = 0$ if $A_0 \not= 0$.

By \cite{K-R} if $e(B, N) = 1$ for every normal subgroup $N$ of $H_i$ of index $p$ then $e(B, H_i) = 1 + dim_{\mathbb{F}_p} H^1(B,  Ind_{H_i}^B(\mathbb{F}_p))$. This together with $e(B, H_i) = 1$ implies $ H^1(B,  Ind_{H_i}^B(\mathbb{F}_p)) = 0$.

By \cite{Sw}  if $e(B, N) = 1$ for every normal subgroup $N$ of $H_i$  such that $H_i / N \simeq \mathbb{Z}$ then  $e(B, H_i) = 1 + rk H^1(B,  Ind_{H_i}^B(\mathbb{Z}))$. Here it is easy to see that $H^1(B,  Ind_{H_i}^B(\mathbb{Z}))$ is $\mathbb{Z}$-torsion-free since if $f \in Der(B,  Ind_{H_i}^B(\mathbb{Z}))$ with $z f  \in PDer(B,  Ind_{H_i}^B(\mathbb{Z}))$ then $f \in PDer(B,  Ind_{H_i}^B(\mathbb{Z}))$. Then $e(B, H_i) = 1$ implies $ H^1(B,  Ind_{H_i}^B(\mathbb{Z})) = 0$.
\end{proof}

\end{document}